\documentclass[12pt]{amsart}
\usepackage{epsfig,comment}
\usepackage{mathtools}
\usepackage[subnum]{cases}
\usepackage{enumerate}
\usepackage{bbm,bm}
\usepackage{amsmath,amssymb,mathrsfs}
\usepackage{color}
\usepackage[bookmarks=true,
bookmarksnumbered=true, breaklinks=true,
pdfstartview=FitH, hyperfigures=false,
plainpages=false, naturalnames=true,
colorlinks=true,pagebackref=true,
pdfpagelabels,linkcolor=blue,
	citecolor=black,
	filecolor=black,
	urlcolor=black]{hyperref}
\usepackage{wasysym}

\newtheorem{theorem}{Theorem}[section]

\newtheorem{lemma}[theorem]{Lemma}
\newtheorem{corollary}[theorem]{Corollary}
\newtheorem{remark}[theorem]{Remark}

\newcommand{\R}{\mathbb{R}}

\newcommand{\Sp}{\mathbb{S}^{d-1}}

\newcommand{\NN}{N^{-\frac{2}{d-1}}}
\newcommand{\MM}{M^{-\frac{2}{d-1}}}

\DeclareMathOperator{\vol}{vol}

\DeclareMathOperator{\Gr}{Gr}

\title[Dual volume approximation of the Euclidean ball]{Dual volume approximation of the Euclidean ball \\by polytopes with a fixed number of $k$-faces}

\author{Steven Hoehner, Carsten Sch\"utt and Elisabeth M. Werner}
\date{\today}

\begin{document}

\setcounter{footnote}{0}

\begin{abstract}\noindent
We study dual volume approximation of the Euclidean ball by polytopes with a prescribed number of $k$-dimensional faces. This continues the authors' previous work on  intrinsic volume approximation of the ball by polytopes with a fixed number of $k$-faces, and develops the corresponding dual radial theory. Our first main result gives a nonasymptotic lower bound for the volume deficit of an inscribed polytope $P_M\subset B_d$ with at most $M$ $k$-faces, for $0\leq k\leq \lfloor d/2\rfloor$. The estimate has the form
\[
\vol_d(B_d\setminus P_M)
\geq \frac{1-e^{-1}}{2}\kappa_d\min\left\{
1,\frac d2\left(\frac{\omega_d}{4\kappa_{d-1}}\right)^{\frac{2}{d-1}}
M^{-\frac{2}{d-1}}\right\},
\]
and, in the vertices case $k=0$, in the large-$M$ regime it recovers the order of the lower bound of Gordon, Reisner and Sch\"utt ({\it J. Approx. Theory}, 1997). We also prove the polar counterpart for the mean width excess of circumscribed polytopes with at most $M$ $k$-faces, for $\lceil d/2\rceil-1\leq k\leq d-1$. More generally, using the analytic extension of the dual volume deviations introduced by Besau, Hoehner and Kur ({\it Int. Math. Res. Not.}, 2021), we
obtain nonasymptotic lower bounds for all $q\in\mathbb R$, including $q=0$, in both the inscribed and circumscribed models. %More generally, using the analytic extension of the dual volume deviations introduced by Besau, Hoehner and Kur ({\it IMRN}, 2021), we obtain nonasymptotic lower bounds for $\widetilde\Delta_q(B_d,P_M)$ for all $q\in\mathbb R$, including the logarithmic case $q=0$.
%Finally, by combining endpoint asymptotic estimates with inequalities for the $f$-vector, we derive asymptotic lower bounds for best dual volume approximation of smooth convex bodies by polytopes with a restricted number of $k$-faces.
\end{abstract}

\maketitle

\renewcommand{\thefootnote}{}

\renewcommand{\thefootnote}{\arabic{footnote}}
\setcounter{footnote}{0}

%%%%%%%%%%%%%%%%%%%%%%%%%%%%%%%%%
\section{Introduction}

The approximation of convex bodies by polytopes is a classical topic in convex and discrete geometry. A central problem is to quantify how well a given convex body can be approximated by a polytope when the combinatorial complexity of the approximating polytope is prescribed. Traditionally, two fundamental models have been studied: approximation by inscribed polytopes with a fixed number of vertices, and approximation by circumscribed polytopes with a fixed number of facets. These two models are dual to one another under the polarity operation, and have been investigated extensively for many natural notions of ``distance'', including the symmetric difference metric, Hausdorff distance, surface area deviation, intrinsic volume metric, and dual volume differences.

\begin{comment}
A celebrated result in this direction is due to Gruber \cite{Gruber-1988}. Let $B_d$ denote the Euclidean unit ball in $\R^d$ ($d\geq 2$) centered at the origin, and let $\mathcal{P}_N^{\rm in}(B_d)$ denote the set of all polytopes inscribed in $B_d$ with at most $N$ vertices.  Then by the result of Gruber, applied in the special case of the Euclidean ball,
\begin{equation}\label{eq:Gruber-asymptotic}
    \lim_{N\to\infty}N^{\frac{2}{d-1}}\inf_{P\in\mathcal{P}_N^{\rm in}(B_d)}\vol_d(B_d\setminus P)=\frac{1}{2}\operatorname{del}_{d-1}(\vol_{d-1}(\partial B_d))^{\frac{d+1}{d-1}}.
\end{equation}
where $\vol_d$ denotes the $d$-dimensional volume, and $\operatorname{del}_{d-1}$ is a positive constant depending only on the dimension; see Remark \ref{}. This result serves as one representative example of the type of results known in the inscribed vertices case, which has been well-studied. 

Similarly, the dual circumscribed facets case has also been well-studied. {\color{red}circumscribed mean width asymptotic formula}
\end{comment}

Much less is known, however, when the complexity of the approximating polytope is measured by the number of $k$-dimensional faces in the intermediate regime $1\leq k\leq d-2$. This  was posed as a problem by Gruber in his book (see \cite[p. 216]{GruberBook}), where he asked for results on the best approximation of convex bodies by polytopes with a restricted number of intermediate-dimensional faces. The difficulty of imposing a bound on the number of intermediate-dimensional faces is of a different nature from the classical vertex and facet restrictions. Vertices and facets are directly tied to the two basic approximation models: inscribed polytopes are controlled naturally by their vertices, while circumscribed polytopes are controlled naturally by their supporting facets. In these \emph{endpoint cases}, the geometry of the approximation is closely connected to coverings of the sphere by caps or to corresponding dual covering problems. By contrast, a restriction on the number of $k$-faces with $1\leq k\leq d-2$ is not directly a local geometric restriction of this kind. A polytope may have relatively few $k$-faces while still having many vertices or many facets, depending on the range of $k$, and the combinatorial information does not immediately translate into a spherical covering or separation estimate. 

In view of this obstruction, one possible strategy for studying intermediate $k$-face approximation involves combining geometric approximation arguments with combinatorial information about the $f$-vector. Very recently, in \cite{HSW-2026} the authors did exactly that and  obtained lower bounds for the best approximation of the Euclidean ball by polytopes with a fixed number of $k$-faces. The main point of that work was that, by combining the endpoint estimates from the cases $k=0,d-1$ with elementary inequalities for the $f$-vector of a polytope, one can extend several known vertex and facet approximation results to a substantial part of the $f$-vector. More precisely, for a polytope $P$ in $\R^d$, denote by $f_k(P)$, $k\in\{0,1,\ldots,d-1\}$, the number of $k$-dimensional faces of $P$, and by $(f_0(P),f_1(P),\ldots,f_{d-1}(P))$ the $f$-vector of $P$. The combinatorial inequalities
\begin{equation}\label{vertices-bd}
f_k(P)\geq f_0(P),\qquad 0\leq k\leq \left\lfloor\frac {d}{2}\right\rfloor,
\end{equation}
and
\begin{equation}\label{facets-bd}
f_k(P)\geq f_{d-1}(P),\qquad \left\lceil\frac {d}{2}\right\rceil-1\leq k\leq d-1,
\end{equation}
which are due to Hinman \cite{Hinman1,Hinman2}, allowed the authors in \cite{HSW-2026} to pass from estimates involving vertices or facets to estimates involving $k$-faces. In \cite{HSW-2026}, this method was combined with arguments involving polar duality and isoperimetricity to obtain lower bounds for intrinsic volume approximation, as well as for the symmetric difference metric, the surface area deviation, and the Hausdorff metric.

The purpose of the present paper is to develop the corresponding theory for radial and dual volume approximation. In contrast to the intrinsic volumes, the dual volumes are naturally expressed in terms of radial functions. In particular, if $K$ is a convex body in $\R^d$ containing the origin in its interior and $j\in\{0,1,\ldots,d\}$, then its $j$th dual volume can be written in the form
\begin{equation}
\widetilde{V}_j(K)
=V_j(B_d)\int_{\Sp}\rho_K(u)^j\,d\sigma(u),
\end{equation}
where $V_j(B_d)$ is the $j$th intrinsic volume of $B_d$ (see Section \ref{background}), $\rho_K(u)=\sup\{\lambda\geq 0:\,\lambda u\in K\}$ is the radial function of $K$, and $\sigma$ denotes normalized surface area measure on $\Sp$. Note that for inscribed polytopes $P\subset B_d$, dual volume approximation is governed by the set of directions in which the radial function $\rho_P$ is noticeably smaller than $1$. This makes the problem geometrically different from the intrinsic volume estimates obtained in \cite{HSW-2026}, where the extended isoperimetric inequality plays a central role.

Our first main result is a nonasymptotic lower bound for the volume difference of the Euclidean ball and an inscribed polytope with a fixed number of $k$-faces. Gordon, Reisner and Sch\"utt \cite{GRS-umbrellas} proved that there are absolute constants $c_1,c_2>0$ such that, for every $d\geq 2$ and every $N\geq (c_1 d)^{(d-1)/2}$, every polytope $P_N\subset B_d$ with at most $N$ vertices satisfies
\begin{equation}\label{eq:GRS-lower}
\vol_d(B_d\setminus P_N)\geq c_2 d\kappa_d N^{-\frac{2}{d-1}}.
\end{equation}
Here and throughout the paper, the Euclidean unit ball in $\R^d$ centered at the origin is denoted $B_d$, and $\kappa_d=\vol_d(B_d)$ denotes its $d$-dimensional volume. We also let $\omega_d=\vol_{d-1}(\partial B_d)$ denote the surface area of $B_d$. Note that $\omega_d=d\kappa_d$. 

Prochno, Sch\"utt and Werner \cite{PSW-2022} asked whether the large-$N$ threshold assumption is really necessary. Since the left-hand side is always bounded above by $\kappa_d$, the displayed estimate cannot hold  for all $N\geq d+1$ with an absolute constant independent of $d$. The natural nonasymptotic formulation should therefore be an estimate of the form
\[
\vol_d(B_d\setminus P_N)
\geq c\kappa_d
\min\left\{1,\frac{d}{2}\left(\frac{\omega_d}{4\kappa_{d-1}}\right)^{\frac{2}{d-1}}N^{-\frac{2}{d-1}}\right\}
\]
for some absolute constant $c>0$. We prove this estimate and, using the combinatorial inequality \eqref{vertices-bd}, we extend it from the vertices case to all $k$-faces with $0\leq k\leq \lfloor d/2\rfloor$. 
In the regime of large $N$, this recovers the order of the lower bound \eqref{eq:GRS-lower}, while the minimum gives the correct nonasymptotic formulation which holds for all admissible $N$.

More generally, we consider the analytic extension of the dual volume deviations introduced by Besau, Hoehner and Kur \cite{BHK}. For $q\neq 0$, these deviations can be expressed as
\[
\widetilde\Delta_q(K,L)
=c_1(d,q)
\int_{\Sp}
|\rho_K(u)^q-\rho_L(u)^q|\,d\sigma(u),
\]
with the normalization
\[
c_1(d,q):=
\begin{cases}
\widetilde{V}_{|q|}(B_d), & 0<|q|\leq d,\\
\kappa_d, & |q|>d.
\end{cases}
\]
Please see Section \ref{subsec:dual-volume-estimates} for the precise definitions and notation. The logarithmic case $q=0$ is described by the corresponding logarithmic radial deviation. We prove that for every inscribed polytope $P_M\subset B_d$ with at most $M$ $k$-faces and with $o\in\operatorname{int}(P_M)$,
\[
\widetilde\Delta_q(B_d,P_M)
\geq \frac{1-e^{-1}}{2}\,c_1(d,q)
\min\left\{1,
\frac{|q|}{2}\left(\frac{\omega_d}{4\kappa_{d-1}}\right)^{\frac{2}{d-1}}
M^{-\frac{2}{d-1}}
\right\}
\]
for all $q\neq 0$, with the analogous logarithmic estimate when $q=0$. Taking $q=d$ recovers the volume estimate above.

Finally, by polarity we obtain the corresponding dual statement for circumscribed polytopes. If $P_M\supset B_d$ has at most $M$ $k$-faces with $\lceil d/2\rceil-1\leq k\leq d-1$, then $P_M^\circ\subset B_d$ has at most $M$ faces of dimension $d-1-k$, which lies in $\{0,\ldots,\lfloor d/2\rfloor\}$. Applying the inscribed result to $P_M^\circ$ yields dual lower bounds for the polar body. In the special case corresponding to mean width, this gives a nonasymptotic lower bound for the mean width difference $w(P_M)-w(B_d)$. 

Thus, the present paper complements \cite{HSW-2026}: while that paper studies intrinsic volume approximation by polytopes with a fixed number of $k$-faces, the present work studies  the dual volume side of the theory.

Let us also mention that in the intermediate-dimensional range, several related results are known, mostly in low-dimensional cases. For $d=3$, asymptotic estimates have been obtained for the approximation of smooth convex bodies by inscribed polytopes with a prescribed number of edges, both for the Hausdorff metric \cite{Boroczky-Fodor-Vigh} and for the symmetric difference metric \cite{Boroczky-Gomis-Tick}. In higher dimensions, the situation is much less complete. For $d\geq 4$, K. J. B\"or\"oczky \cite{Boroczky-2000} established lower estimates, in the symmetric difference metric and the Hausdorff distance, for the approximation of smooth convex bodies by arbitrarily positioned polytopes while bounding the number of $k$-faces. In the special case of the Euclidean ball, B\"or\"oczky's result for the symmetric difference metric  was recently improved in \cite{HSW-2026} by a factor of dimension, for half of the $f$-vector of the approximating polytope. See also \cite{HSW-2026} for an analogous result for the surface area deviation of the Euclidean ball and an arbitrarily positioned polytope with a prescribed number of $k$-faces. Related asymptotic lower bounds for the intrinsic volume approximation of the Euclidean ball by inscribed and circumscribed polytopes with a prescribed number of $k$-faces were proved by Besau, Hoehner and Kur \cite{BHK}, under the assumption that the approximating polytopes are simplicial. Gruber also raised closely connected questions concerning the typical faces of best-approximating polytopes in asymptotic approximation problems \cite{GruberIV,Gruber2001}. Further contributions in this direction, especially for approximation under restrictions on the number of edges, can be found in \cite{BoroczkyTickWintsche2007,LRSW,Vigh2009}.
%%%%%%%%%%%%%%%%%%%%%%%%%%%%%%%%%%

\subsection{Overview of the paper}

The paper is organized as follows. In Section \ref{background}, we recall the background and notation that will be used throughout the paper. In Section \ref{mainresults-sec}, we state the main  lower bounds, including the volume deficit estimate, its polar mean width consequence, and the analytic dual volume estimate. The proofs of our main results are given in Section \ref{proofs-sec}. Then, in Section \ref{sec:upper-bds}, we prove that the lower bounds in Theorem \ref{inscribed-vol-kfaces} for $k=0$, and in \ref{circumscribed-mw-kfaces} for $k=d-1$, are sharp in the large-$N$ regime, up to absolute constants. Finally, in Section \ref{sec-discussion}, we conclude the paper with a brief discussion of possible extensions and related asymptotic questions.

%%%%%%%%%%%%%%%%%%%%%%%%

\section{Background and notation}\label{background}

We shall work in the $d$-dimensional Euclidean space
\[
\R^d=\{(x_1,\ldots,x_d):\,x_1,\ldots,x_d\in\R\}.
\]
Vectors in $\R^d$ are denoted by lowercase Roman letters, e.g., $x=(x_1,\ldots,x_d)\in\R^d$. The origin in $\R^d$ is denoted by $o$. For $x=(x_1,\ldots,x_d),y=(y_1,\ldots,y_d)\in\R^d$, the standard inner product on $\R^d$ is given by $\langle x,y\rangle=\sum_{i=1}^d x_i y_i$, 
and the Euclidean norm of $x$ is $\|x\|_2=\sqrt{\langle x,x\rangle}=\sqrt{\sum_{i=1}^d x_i^2}$. The $d$-dimensional Euclidean unit ball is $B_d=\{x\in\R^d:\,\|x\|_2\leq 1\}$. Its $d$-dimensional volume is given by $\kappa_d=\vol_d(B_d)=\frac{\pi^{d/2}}{\Gamma(\frac{d}{2}+1)}$, where $\Gamma(x)=\int_0^\infty t^{x-1}e^{-t}\,dt$ is the Gamma function ($x>0$).

Let $\mathcal{K}_o^d$ denote the class of convex bodies in $\R^d$ which contain the origin in their interiors. The dual  Brunn--Minkowski theory, introduced by Lutwak \cite{Lutwak-1975,Lutwak-1988,Lutwak-1990}, arises when one replaces the classical Minkowski sum of convex bodies by the radial sum.  For $K,L\in\mathcal{K}_o^d$, the \emph{radial sum} $K\widetilde{+}L$ is the body defined by
\[
\rho_{K\widetilde{+}L}=\rho_K+\rho_L.
\]
The radial sum is not necessarily convex, but it is always a star body. The \emph{dual Steiner formula} states that, for $K\in\mathcal{K}_o^d$ and $r>0$, the volume of the radial parallel body $K\widetilde{+}rB_d$ is a polynomial in $r$:
\begin{equation}\label{eq:dual-steiner}
\vol_d(K\widetilde{+}rB_d)=\sum_{j=0}^d\kappa_j\widetilde{V}_{d-j}(K)r^j.
\end{equation}
The coefficients $\widetilde{V}_0(K),\ldots,\widetilde{V}_d(K)$ defined implicitly in \eqref{eq:dual-steiner} are called the \emph{dual volumes} of $K$. Note that, in particular,  $\widetilde{V}_d(K)=\vol_d(K)$ is the usual $d$-dimensional volume of $K$. Also, in this normalization, we have $\widetilde{V}_j(B_d)=V_j(B_d)$ for $j\in\{0,1,\ldots,d\}$ where $V_j(B_d)={d\choose j}\frac{\kappa_d}{\kappa_{d-j}}$ is the $j$th intrinsic volume of $B_d$.

Using polar coordinates, we can express the dual volume as
\[
\widetilde{V}_j(K) = V_j(B_d)\int_{\Sp}\rho_K(u)^j\,d\sigma(u).
\]
Lutwak \cite{Lutwak-1979} proved the following ``dual Kubota'' formula:
\[
\widetilde{V}_j(K)={d\choose j}\frac{\kappa_d}{\kappa_j\kappa_{d-j}}\int_{\Gr(d,j)}\vol_j(K\cap E)\,d\nu_j(E).
\]
Here $\Gr(d,j)$ denotes the Grassmannian of all $j$-dimensional subspaces of $\R^d$, and $\nu_j$ denotes the unique Haar probability measure on $\Gr(d,j)$. For more background on dual volumes of convex bodies, we refer the reader to, e.g., \cite{gardner2006geometric,SchneiderBook}.

Let $\delta$ be a distance functional on the set of convex bodies in $\R^d$. For a convex body $K$ in $\R^d$ and a class $\mathcal{C}$ of polytopes in $\R^d$, we set 
\[
\delta(K,\mathcal{C}):=\inf_{P\in\mathcal{C}}\delta(K,P).
\]
When a polytope achieving the infimum exists, it is called a \emph{best-approximating polytope} of $K$ with respect to the functional $\delta$. For a convex body $K\subset\R^d$ and $N\in\mathbb{N}$, we set
\[
\mathcal{P}_N^{\rm in}(K)
:=\{P\subset K:P\text{ is a polytope and }f_0(P)\leq N\},
\]
and
\[
\mathcal{P}_N^{\rm out}(K)
:=\{P\supset K:P\text{ is a polytope and }f_{d-1}(P)\leq N\}.
\]
For background on the results on the approximation of convex bodies by polytopes which are most closely related to this paper, we refer the reader to, e.g., \cite{BH-2022,BHK,Bronshtein-survey,Hoehner-survey,HSW-2026,Hoehner2016,PSoSchW-2025,PSW-2022,Reitzner2002,SW:2003} and the references therein.

Throughout the paper, we let $c,C,c_1,C_2$, etc. denote positive absolute constants. If a constant depends on a set of parameters, then that dependence will always be expressed explicitly, e.g., $C(d)$ is a constant that depends only on $d$, and $c_2(d,q)$ is a constant that depends only on $d$ and $q$.

%%%%%%%%%%%%%%%%%%%%%%%%%%%%%
\section{Main Results}\label{mainresults-sec}

\subsection{General nonasymptotic lower bounds}

Our first result is a lower bound for the volume difference of the Euclidean ball and an inscribed polytope with a fixed number of $k$-faces.

\begin{theorem}\label{inscribed-vol-kfaces}
    Let $d\geq 2$ and $k\in\{0,1,\ldots,\lfloor\frac{d}{2}\rfloor\}$. For all polytopes $P_M\subset B_d$ with at most $M$ $k$-faces, we have
    \[
    \vol_d(B_d\setminus P_M) \geq \frac{1-e^{-1}}{2}\kappa_d
    \min\left\{
    1, \frac d2\left(\frac{\omega_d}{4\kappa_{d-1}}\right)^{\frac{2}{d-1}}\MM
    \right\}.
    \]
\end{theorem}
Thus, Theorem \ref{inscribed-vol-kfaces} identifies the capped nonasymptotic formulation of the question of Prochno, Sch\"utt and Werner \cite{PSW-2022} which was mentioned in the introduction, and resolves  the refined conjecture affirmatively. 

\begin{remark}
In the vertices case $k=0$, Theorem \ref{inscribed-vol-kfaces} recovers, up to absolute constants, the lower bound of Gordon, Reisner and Sch\"utt \cite{GRS-umbrellas} in the large-$M$ regime. Indeed, if $M\geq (c_1d)^{\frac{d-1}{2}}$ for a sufficiently large absolute constant $c_1>0$, then
\[
\frac{d}{2}\left(\frac{\omega_d}{4\kappa_{d-1}}\right)^{\frac{2}{d-1}}\MM\leq 1.
\]
Hence the minimum in Theorem \ref{inscribed-vol-kfaces} is attained by its second term, and we obtain
\[
\vol_d(B_d\setminus P_M)
\geq cd\kappa_d\MM,
\]
where $c>0$ is an absolute constant. This is the same order as the lower estimate of Gordon, Reisner and Sch\"utt for inscribed polytopes with $M$ vertices.
\end{remark}

Recall that for a convex body $K$ in $\R^d$, the \emph{mean width} of $K$ is defined by 
\[
w(K)=2\int_{\Sp}h_K(u)\,d\sigma(u),
\]
where $h_K(u)=\sup_{x\in K}\langle x,u\rangle$ is the support function of $K$ in the direction $u\in\Sp$. Our next result is a general lower bound for the mean width difference of the Euclidean ball and a circumscribed polytope with a fixed number of $k$-faces. 

\begin{theorem}\label{circumscribed-mw-kfaces}
    Let $d\geq 2$ and $k\in\{\lceil d/2\rceil-1,\ldots,d-1\}$. For all polytopes $P_M\supset B_d$ with at most $M$ $k$-faces, we have
    \[
    w(P_M)-w(B_d) \geq \frac{1-e^{-1}}{d}
    \min\left\{
    1, \frac d2\left(\frac{\omega_d}{4\kappa_{d-1}}\right)^{\frac{2}{d-1}}\MM
    \right\}.
    \]
\end{theorem}

\begin{remark}
By Stirling's formula, the dimensional constant appearing in Theorems \ref{inscribed-vol-kfaces} and \ref{circumscribed-mw-kfaces} satisfies
\[
\left(\frac{\omega_{d}}{4 \kappa_{d-1}}\right)^{\frac{2}{d-1}}=1+O\left(\frac{\ln d}{d}\right).
\]
\end{remark}

\begin{remark}
    Polar dual analogues of Theorems \ref{inscribed-vol-kfaces} and \ref{circumscribed-mw-kfaces} were proved in \cite{HSW-2026}. No minimum appears in their results. This is because in their result for volume, the polytope was circumscribed, so no contradiction arises in that case for small $M$, and in their inscribed mean width result, since $w(B_d)-w(P_M)\leq 2$, no contradiction arises in that case either. This is why the minimum must appear in Theorem \ref{inscribed-vol-kfaces} but not in the results of \cite{HSW-2026}. The minimum in Theorem \ref{circumscribed-mw-kfaces} appears as an artifact of our proof method, since we use the inequality in Theorem \ref{inscribed-vol-kfaces} to obtain this bound.
\end{remark}

\begin{remark}
    The preceding theorems are stated in the slightly sharper form that follows directly from the proof. If one prefers to suppress the factor $1/2$ inside the minimum, then the elementary inequality
    \[
        \forall a\geq 0,\qquad \min\{1,a/2\}\geq \frac{1}{2}\min\{1,a\},
    \]
    shows, for example, that Theorem \ref{inscribed-vol-kfaces} implies the slightly less sharp--but more compact--estimate
    \[
        \vol_d(B_d\setminus P_M)
        \geq\frac{1-e^{-1}}{4}\kappa_d
        \min\left\{
        1, d\left(\frac{\omega_d}{4\kappa_{d-1}}\right)^{\frac{2}{d-1}}\MM
        \right\}.
    \]
    Similarly, Theorem \ref{circumscribed-mw-kfaces} implies the corresponding estimate with the same replacement of $\frac d2$ by $d$, at the cost of an additional factor $1/2$ in the leading constant.
\end{remark}

%%%%%%%%%%%%%%%%%%%%%%%%%%%

\subsection{Dual volume approximation}\label{subsec:dual-volume-estimates}

We will obtain Theorem \ref{inscribed-vol-kfaces} from a more general theorem on dual volume approximation of the Euclidean ball.

Let $q\in\R$. The following analytic extension of the dual volumes was defined in \cite{BHK}. First, for $q\in[0,d]$, set
\[
    \widetilde{V}_q(B_d):=
    \pi^{\frac q2}
    \frac{\Gamma(d+1)}{\Gamma(q+1)\Gamma(d-q+1)}
    \frac{\Gamma\left(\frac{d-q}{2}+1\right)}
         {\Gamma\left(\frac d2+1\right)}.
\]
In particular, for $q=j\in\{0,1,\ldots,d\}$, this agrees with the usual intrinsic volume $V_j(B_d)$. For $K\in\mathcal K_o^d$ and $q\in\mathbb R$, define
\[
    \widetilde{V}_q(K):=
    \begin{cases}
    \displaystyle
    \widetilde{V}_{|q|}(B_d)
    \int_{\Sp}\rho_K(u)^q\,d\sigma(u),
    & \text{if } q\in[-d,d],\\[2ex]
    \displaystyle
    \kappa_d
    \int_{\Sp}\rho_K(u)^q\,d\sigma(u),
    & \text{if } q\in\mathbb R\setminus[-d,d].
    \end{cases}
\]

For $q\in\R\setminus\{0\}$, set
\[
    c_1(d,q):=\begin{cases}
        \widetilde{V}_{|q|}(B_d), &0<|q|\leq d,\\
        \kappa_d, & |q|>d.
    \end{cases}
\]
For $K,L\in\mathcal{K}_o^d$, we define the \emph{$q$th dual volume deviation} by
\begin{equation}\label{eq:q-dual-deviation-definition}
\widetilde{\Delta}_q(K,L):=c_1(d,q)\int_{\Sp}|\rho_K(u)^q-\rho_L(u)^q|\,d\sigma(u).
\end{equation}
Equivalently, for $q>0$ this can be written as
\begin{equation}\label{eq:q-positive-deviation-set-formula}
\widetilde{\Delta}_q(K,L)=\widetilde{V}_q(K)+\widetilde{V}_q(L)-2\widetilde{V}_q(K\cap L),
\end{equation}
whereas for $q<0$ it can be written as
\begin{equation}\label{eq:q-negative-deviation-set-formula}
\widetilde{\Delta}_q(K,L)=2\widetilde{V}_q(K\cap L)-\widetilde{V}_q(K)-\widetilde{V}_q(L).
\end{equation}
Indeed, if $q<0$, then the map $t\mapsto t^q$ is decreasing, so the sign in the set-theoretic formula is reversed relative to the case $q>0$.

As in \cite{BHK}, for $q=0$ we define
\[
\widehat{V}_0(K):=\int_{\Sp}\ln\rho_K(u)\,d\sigma(u).
\]
For $K,L\in\mathcal{K}_o^d$, the \emph{0th dual volume deviation} is defined by
\begin{equation}\label{eq:zero-dual-deviation-definition}
    \widehat{\Delta}_0(K,L):=\widehat{V}_0(K)+\widehat{V}_0(L)-2\widehat{V}_0(K\cap L).
\end{equation}

Next, we present our main results on dual volume approximation of the Euclidean ball by polytopes. The first result is for inscribed approximation.

\begin{theorem}\label{q-dual-volume-kfaces}
Let $d\geq 2$ and let $k\in\{0,1,\ldots,\lfloor d/2\rfloor\}$. For every polytope $P_M\subset B_d$ with at most $M$ $k$-faces and with $o\in\operatorname{int}(P_M)$, the following hold.

\begin{itemize}
    \item[(i)] If $0<|q|\leq d$, then
    \[
    \widetilde\Delta_q(B_d,P_M)
    \geq
    \frac{1-e^{-1}}{2}\widetilde{V}_{|q|}(B_d)
    \min\left\{
    1, \frac{|q|}{2}\left(\frac{\omega_d}{4\kappa_{d-1}}\right)^{\frac{2}{d-1}}
    M^{-\frac{2}{d-1}}
    \right\}.
    \]

    \item[(ii)] If $|q|>d$, then
    \[
    \widetilde\Delta_q(B_d,P_M)
    \geq
    \frac{1-e^{-1}}{2}\kappa_d
    \min\left\{
    1, \frac{|q|}{2}\left(\frac{\omega_d}{4\kappa_{d-1}}\right)^{\frac{2}{d-1}}
    M^{-\frac{2}{d-1}}
    \right\}.
    \]

    \item[(iii)] If $q=0$, then
    \[
    \widehat\Delta_0(B_d,P_M)
    \geq
    \frac{1}{4}\left(\frac{\omega_d}{4\kappa_{d-1}}\right)^{\frac{2}{d-1}}
    M^{-\frac{2}{d-1}}.
    \]
\end{itemize}
\end{theorem}

The proof is elementary and is based on a sphere covering argument. Let us briefly describe the main ideas. If $P_M\subset B_d$ has at most $M$ vertices, then there is a set of directions whose spherical measure is bounded below by a positive absolute constant and for which $\rho_{P_M}(u)\leq 1-\eta_{d,M}$, where 
$\eta_{d,M}\asymp
\left(\frac{\omega_d}{\kappa_{d-1}M}\right)^{2/(d-1)}$. 
Integrating the radial deficit over this set yields the desired volume estimate. The same argument gives us, without additional work, lower bounds for all analytic dual volume deviations.

The other main result of this section concerns dual volume approximation of the ball by circumscribed polytopes.

\begin{theorem}\label{q-dual-volume-circ-kfaces}
Let $d\geq 2$ and let $k\in\{\lceil d/2\rceil-1,\ldots,d-1\}$. For every polytope $Q_M\supset B_d$ with at most $M$ $k$-faces, the following inequalities hold:

\begin{itemize}
\item[(i)] If $0<|q|\leq d$, then
\[
\widetilde\Delta_q(B_d,Q_M)
\geq\frac{3(1-e^{-1})}{4}\widetilde{V}_{|q|}(B_d)
\min\left\{1,\frac{|q|}{6}
\left(\frac{\omega_d}{4\kappa_{d-1}}\right)^{\frac{2}{d-1}}
M^{-\frac{2}{d-1}}
\right\}.
\]

\item[(ii)] If $|q|>d$, then
\[
\widetilde\Delta_q(B_d,Q_M)
\geq\frac{3(1-e^{-1})}{4}\kappa_d
\min\left\{1,\frac{|q|}{6} \left(\frac{\omega_d}{4\kappa_{d-1}}\right)^{\frac{2}{d-1}} M^{-\frac{2}{d-1}} \right\}.
\]

\item[(iii)] If $q=0$, then
\[
\widehat\Delta_0(B_d,Q_M)
\geq\frac{1}{8}\left(\frac{\omega_d}{4\kappa_{d-1}}\right)^{\frac{2}{d-1}} M^{-\frac{2}{d-1}}.
\]
\end{itemize}
\end{theorem}

\section{Proofs of main results}\label{proofs-sec}

We will need the following elementary inequality.

\begin{lemma}\label{lem:ineq}
    For all $t\in[0,1]$ and all $a>0$, we have
    \begin{equation}\label{eq:ineq-tbd}
        1-(1-t)^a\geq (1-e^{-1})\min\{1,at\}.
    \end{equation} 
\end{lemma}

\begin{proof}
    For $t\in[0,1)$, define the function $f(t)=\ln(1-t)+t$. Its derivative satisfies
    \[
f'(t)=-\frac{1}{1-t}+1=-\frac{t}{1-t}\leq 0,
    \]
    so $f$ is decreasing. Since $f(0)=0$, this implies $f(t)\leq 0$, i.e., $\ln(1-t)+t\leq 0$, or, equivalently, $\ln(1-t)\leq -t$ for all $t\in[0,1)$. At $t=1$, we have $\ln(1-t)=\ln 0=-\infty$, so the inequality still holds in the extended sense. Hence, the inequality $\ln(1-t)\leq -t$ holds for all $t\in[0,1]$. Multiplying both sides by $a>0$, we get $a\ln(1-t)\leq -at$, and hence $(1-t)^a\leq e^{-at}$ for all $t\in[0,1]$ and $a>0$. Therefore,
    \begin{equation}\label{eq:step1}
      \forall t\in[0,1],\quad\forall a>0,\qquad  1-(1-t)^a\geq 1-e^{-at}.
    \end{equation}

    Now set $x=at\geq 0$. We will establish the following inequality: 
    \begin{equation}\label{eq:lem-main-step}
        \forall x\geq 0,\qquad 1-e^{-x} \geq (1-e^{-1})\min\{1,x\}.
    \end{equation}
    For the interval $0\leq x\leq 1$, the function $g(x)=1-e^{-x}$ is concave. Therefore, its graph lies above the chord joining $(0,g(0))$ and $(1,g(1))$. Thus, in this case we get that for every $x\in[0,1]$,
    \begin{align*}
1-e^{-x} = g(x)&=g(x\cdot 1+(1-x)\cdot 0)\\
&\geq xg(1)+(1-x)g(0)=(1-e^{-1})x\geq (1-e^{-1})\min\{1,x\}. 
    \end{align*}
    Now assume that $x\geq 1$. Since $g(x)$ is an increasing function, for every $x\geq 1$ we have
    \[
g(x)=1-e^{-x} \geq 1-e^{-1}=(1-e^{-1})\cdot 1\geq (1-e^{-1})\min\{1,x\}.
    \]
    Combining the two cases, we have proved  \eqref{eq:lem-main-step}. Now substituting $x=at\geq 0$ in \eqref{eq:lem-main-step}, we get
    \[
\forall t\in[0,1],\quad\forall a>0,\qquad 1-e^{-at} \geq (1-e^{-1})\min\{1,at\}.
    \]
    Combining this with inequality \eqref{eq:step1}, we obtain the desired inequality \eqref{eq:ineq-tbd}.
\end{proof}

\subsection{Proof of Theorem \ref{q-dual-volume-kfaces}}

In the first step, we begin with the case $k=0$. In the second step, we extend the result to the desired range of $k$.

Assume first that $k=0$, so that $f_0(P_M)\leq M$. Since $o\in\operatorname{int}(P_M)$, the polytope $P_M$ is full-dimensional, and hence $M\geq d+1$. Denote the vertices of $P_M$ by $v_1,\ldots,v_m$, where $m\leq M$, and set
\[
\eta_{d,M}:=\frac{1}{2}\left(\frac{\omega_d}{4\kappa_{d-1}M}\right)^{\frac{2}{d-1}}.
\]
By Stirling's inequality and the fact that $M\geq d+1$, we have $\eta_{d,M}\leq 1/2$. For $u\in\Sp$, define
\[
\mathcal{V}_{P_M}:=\{u\in\Sp:\,\rho_{P_M}(u)\leq 1-\eta_{d,M}\}.
\]
We claim that
\[
\vol_{d-1}(\mathcal{V}_{P_M})\geq \frac{1}{2}\omega_d.
\]
Indeed, if $u\not\in\mathcal{V}_{P_M}$, then $\rho_{P_M}(u)>1-\eta_{d,M}$. 
Since $\rho_{P_M}(u)u\in P_M$, it can be expressed as a convex combination of the vertices of $P_M$. Hence, there exist coefficients $\lambda_1,\ldots,\lambda_m\geq 0$ with $\sum_{i=1}^m\lambda_i=1$ such that
\[
\rho_{P_M}(u)u=\sum_{i=1}^m\lambda_i v_i.
\]
Taking the inner product with $u$, we obtain
\[
\rho_{P_M}(u)=\sum_{i=1}^m\lambda_i\langle u,v_i\rangle.
\]
Since $\rho_{P_M}(u)>1-\eta_{d,M}$, there must exist an index $i$ such that
\[
\langle u,v_i\rangle>1-\eta_{d,M}.
\]
In particular, $v_i\neq o$. Since $\|v_i\|_2\leq 1$, it follows that
\[
\left\langle u,\frac{v_i}{\|v_i\|_2}\right\rangle>1-\eta_{d,M}.
\]
Thus, $u$ belongs to the spherical cap
\[
C_i:=\left\{u\in\Sp:
\left\langle u,\frac{v_i}{\|v_i\|_2}\right\rangle>1-\eta_{d,M}
\right\}.
\]
Consequently,
\[
\Sp\setminus\mathcal{V}_{P_M}
\subset\bigcup_{i=1}^m C_i.
\]
Using the standard estimate for a spherical cap of height $\eta_{d,M}$, namely,
\[
\vol_{d-1}(C_i)\leq 2\kappa_{d-1}(2\eta_{d,M})^{\frac{d-1}{2}},
\]
we derive that
\[
\vol_{d-1}(\Sp\setminus\mathcal{V}_{P_M})
\leq 2M\kappa_{d-1}(2\eta_{d,M})^{\frac{d-1}{2}}
=\frac{1}{2}\omega_d.
\]
Therefore,
\[
\vol_{d-1}(\mathcal{V}_{P_M})=\vol_{d-1}(\Sp)-\vol_{d-1}(\Sp\setminus\mathcal{V}_{P_M})\geq \frac{1}{2}\omega_d.
\]
Equivalently, letting $\sigma$ denote the normalized spherical measure on $\Sp$, we have $\sigma(\mathcal{V}_{P_M})\geq \frac{1}{2}$. 

We first consider the case $q\neq 0$. If $q>0$, then for every $u\in\mathcal{V}_{P_M}$,
\[
1-\rho_{P_M}(u)^q
\geq 1-(1-\eta_{d,M})^q.
\]
If $q<0$, then writing $|q|=-q$, for every $u\in\mathcal{V}_{P_M}$ we have
\[
\rho_{P_M}(u)^q-1
=\rho_{P_M}(u)^{-|q|}-1
\geq(1-\eta_{d,M})^{-|q|}-1
\geq 1-(1-\eta_{d,M})^{|q|}.
\]
Thus, in either case, we obtain the lower estimate
\[
\forall u\in\mathcal{V}_{P_M},\quad |1-\rho_{P_M}(u)^q|
\geq 1-(1-\eta_{d,M})^{|q|}.
\]
Next, recall the estimate
\[
\forall t\in[0,1],\quad\forall a>0,\qquad
1-(1-t)^a\geq (1-e^{-1})\min\{1,at\},
\]
from Lemma \ref{lem:ineq}. Applying this with $a=|q|$ and $t=\eta_{d,M}$, we obtain
\[
\forall u\in\mathcal{V}_{P_M},\quad |1-\rho_{P_M}(u)^q|
\geq (1-e^{-1})\min\{1,|q|\eta_{d,M}\}.
\]
This implies
\begin{align*}
\int_{\Sp}|1-\rho_{P_M}(u)^q|\,d\sigma(u)
&\geq\int_{\mathcal{V}_{P_M}}|1-\rho_{P_M}(u)^q|\,d\sigma(u)\\
&\geq (1-e^{-1})\sigma(\mathcal{V}_{P_M})\min\{1,|q|\eta_{d,M}\}\\
&\geq\frac{1-e^{-1}}{2}
\min\left\{1, \frac{|q|}{2}\left(\frac{\omega_d}{4\kappa_{d-1}M}\right)^{\frac{2}{d-1}}
\right\}.
\end{align*}
Now suppose that $0<|q|\leq d$. By the definition of the analytic extension of the dual volume deviation,
\[
\widetilde\Delta_q(B_d,P_M)
=\widetilde{V}_{|q|}(B_d)\int_{\Sp}|1-\rho_{P_M}(u)^q|\,d\sigma(u).
\]
Hence,
\[
\widetilde\Delta_q(B_d,P_M)
\geq
\frac{1-e^{-1}}{2}\widetilde{V}_{|q|}(B_d)
\min\left\{1,
\frac{|q|}{2}\left(\frac{\omega_d}{4\kappa_{d-1}}\right)^{\frac{2}{d-1}}
M^{-\frac{2}{d-1}}
\right\}.
\]
This proves the assertion for $0<|q|\leq d$ in the vertex case.

Next, suppose that $|q|>d$. By the definition of the dual volume deviation in this range, we have
\[
\widetilde\Delta_q(B_d,P_M)
=\kappa_d\int_{\Sp}|1-\rho_{P_M}(u)^q|\,d\sigma(u).
\]
Therefore,
\[
\widetilde\Delta_q(B_d,P_M)
\geq
\frac{1-e^{-1}}{2}\kappa_d
\min\left\{1,
\frac{|q|}{2}\left(\frac{\omega_d}{4\kappa_{d-1}}\right)^{\frac{2}{d-1}}
M^{-\frac{2}{d-1}}
\right\}.
\]
This proves the assertion for $|q|>d$ in the vertex case.

It remains to prove the logarithmic case $q=0$. By definition,
\[
\widehat\Delta_0(B_d,P_M)
=\int_{\Sp}|\ln \rho_{P_M}(u)|\,d\sigma(u),
\]
since $\rho_{B_d}(u)=1$ for every $u\in\Sp$. For every $u\in\mathcal{V}_{P_M}$, we have
\[
|\ln \rho_{P_M}(u)|= -\ln \rho_{P_M}(u)
\geq -\ln(1-\eta_{d,M})
\geq\eta_{d,M}.
\]
Thus,
\[
\widehat\Delta_0(B_d,P_M)
\geq\sigma(\mathcal{V}_{P_M})\eta_{d,M}
\geq\frac{1}{2}\eta_{d,M}
=\frac{1}{4}\left(\frac{\omega_d}{4\kappa_{d-1}M}\right)^{\frac{2}{d-1}}.
\]
This yields the desired lower bound in the case $q=0$.

Finally, let $P_M\subset B_d$ have at most $M$ $k$-faces, where $k\in\{0,1,\ldots,\lfloor d/2\rfloor\}$. By \eqref{vertices-bd}, we have $f_0(P_M)\leq f_k(P_M)\leq M$. 
Thus, $P_M$ has at most $M$ vertices. Applying the vertex case proved above gives the desired lower estimate. This completes the proof. \qed

%%%%%%%%%%%%%%%%%%%%%%%%%%%%%%%%%%%
\subsection{Proof of Theorem \ref{q-dual-volume-circ-kfaces}}

We first prove the theorem in the facets case $k=d-1$. Let $Q_M\in\mathcal{P}_M^{\rm out}(B_d)$. 
Set $\widetilde{\eta}_{d,M}:=\frac{2}{3}\eta_{d,M}$, and define
\[
\mathcal U_{Q_M}
:=\left\{u\in\Sp:\, \rho_{Q_M}(u)>1+\widetilde{\eta}_{d,M}\right\}.
\]
As shown in \cite{HSW-2026},
\[
\sigma(\mathcal{U}_{Q_M})\geq \frac{3}{4}.
\]

First, suppose that $q\neq 0$. Since $Q_M\supset B_d$, we have $\rho_{Q_M}(u)\geq 1$ for every $u\in\Sp$. 
Hence
\[
\begin{aligned}
\widetilde\Delta_q(B_d,Q_M)
&=c_1(d,q)\int_{\Sp}
|\rho_{Q_M}(u)^q-1|\,d\sigma(u)\\
&\geq c_1(d,q)\int_{\mathcal U_{Q_M}}
|\rho_{Q_M}(u)^q-1|\,d\sigma(u)\\
&\geq c_1(d,q)\bigl|(1+\widetilde{\eta}_{d,M})^q-1\bigr|
\sigma(\mathcal U_{Q_M})\\
&\geq\frac{3}{4}c_1(d,q)
\bigl|(1+\widetilde{\eta}_{d,M})^q-1\bigr|.
\end{aligned}
\]
We now estimate the last factor. Set $a=|q|>0$. We claim that for every $t\in[0,1]$,
\[
|(1+t)^q-1|\geq 1-(1+t)^{-a}.
\]
Indeed, if $q<0$, this is an equality; if $q>0$, it follows from
\[
(1+t)^q-1\geq 1-(1+t)^{-q}.
\]
Moreover, using the inequality $\ln(1+t)\geq\frac{t}{2}$, $0\leq t\leq 1$, we get
\[
(1+t)^{-a}
=e^{-a\ln(1+t)}
\leq e^{-at/2}.
\]
Consequently, $|(1+t)^q-1|
\geq 1-e^{-at/2}$. Using again the inequality $1-e^{-x}\geq (1-e^{-1})\min\{1,x\}$, $x\geq 0$, with $x=at/2$, we obtain
\[
|(1+t)^q-1|
\geq(1-e^{-1})\min\left\{1,\frac{|q|t}{2}\right\}.
\]
Applying this with $t=\widetilde\eta_{d,M}$ and using the definition of $\widetilde\eta_{d,M}$,  
we obtain
\[
\bigl|(1+\widetilde{\eta}_{d,M})^q-1\bigr|\geq(1-e^{-1})\min\left\{
1,\frac{|q|}{6}\left(\frac{\omega_d}{4\kappa_{d-1}}\right)^{\frac{2}{d-1}}
M^{-\frac{2}{d-1}}\right\}.
\]
Therefore,
\[
\widetilde\Delta_q(B_d,Q_M)
\geq\frac{3(1-e^{-1})}{4}c_1(d,q)
\min\left\{1,\frac{|q|}{6}
\left(\frac{\omega_d}{4\kappa_{d-1}}
\right)^{\frac{2}{d-1}}
M^{-\frac{2}{d-1}}\right\}.
\]
This proves parts (i) and (ii) in the case $k=d-1$.

It remains to consider the logarithmic case $q=0$. Since $Q_M\supset B_d$,
\[
\widehat\Delta_0(B_d,Q_M)
=\int_{\Sp}\ln\rho_{Q_M}(u)\,d\sigma(u).
\]
Therefore
\[
\begin{aligned}
\widehat\Delta_0(B_d,Q_M)
&\geq\int_{\mathcal{U}_{Q_M}}
\ln\rho_{Q_M}(u)\,d\sigma(u)
\geq\ln(1+\widetilde{\eta}_{d,M})\sigma(\mathcal{U}_{Q_M})
\geq\frac{3}{4}\ln(1+\widetilde{\eta}_{d,M}).
\end{aligned}
\]
Since $0\leq\widetilde{\eta}_{d,M}\leq 1$, we have
\[
\ln(1+\widetilde{\eta}_{d,M})
\geq\frac{1}{2}\widetilde{\eta}_{d,M}.
\]
Hence
\[
\widehat\Delta_0(B_d,Q_M)
\geq\frac{3}{8}\widetilde\eta_{d,M}.
\]
Now using the definition of $\widetilde{\eta}_{d,M}$, we obtain
\[
\widehat\Delta_0(B_d,Q_M)
\geq\frac{1}{8}\left(\frac{\omega_d}{4\kappa_{d-1}}\right)^{\frac{2}{d-1}}
M^{-\frac{2}{d-1}}.
\]
This proves part (iii) in the case $k=d-1$.

Finally, suppose that $k\in
\{\left\lceil\frac{d}{2}\right\rceil-1,\ldots,d-1\}$
and that $Q_M\supset B_d$ has at most $M$ $k$-faces. By \eqref{facets-bd}, we have $f_{d-1}(Q_M)\leq f_k(Q_M)\leq M$, so $Q_M$ has at most $M$ facets. Applying the facets case proved above gives the desired estimates. This completes the proof. \qed

%%%%%%%%%%%%%%%%%%%%%%%%%%%%%%%%%%%

\subsection{Proof of Theorem \ref{inscribed-vol-kfaces}}

Choosing $q=d$ in Theorem \ref{q-dual-volume-kfaces}, we obtain the desired inequality in the case $o\in\operatorname{int}(P_M)$. It remains to consider the case when $o\not\in\operatorname{int}(P_M)$. Then $P_M$ is contained in a half-ball of $B_d$. Thus, by the separation theorem, there exists a direction  $u\in\Sp$ such that
\[
P_M\subset B_d\cap\{x\in\R^d:\,\langle x,u\rangle\leq 0\}.
\]
Therefore,
\[
    \vol_d(B_d\setminus P_M)\geq \frac{1}{2}\kappa_d.
\]
Since
\[
    \frac{1}{2}\kappa_d
    \geq
    \frac{1-e^{-1}}{2}\kappa_d
    \min\left\{
    1,\,
    \frac d2\left(\frac{\omega_d}{4\kappa_{d-1}}\right)^{\frac{2}{d-1}}\MM
    \right\},
\]
this implies the asserted bound. \qed
%%%%%%%%%%%%%%%%%%%%%%%%%%%%%%%

\subsection{Proof of Theorem \ref{circumscribed-mw-kfaces}}

Let $P_M\supset B_d$ be a polytope with at most $M$ $k$-faces, where $k\in\{\lceil d/2\rceil -1,\ldots,d-1\}$. Since $P_M\supset B_d$ is bounded, the polar polytope $P_M^\circ\subset B_d$ contains the origin in its interior. Moreover, since $o\in\operatorname{int}(P_M)$, polarity gives an inclusion-reversing bijection between the $k$-faces of $P_M$ and the $(d-1-k)$-faces of $P_M^\circ$. Hence
\[
    f_k(P_M)=f_{d-1-k}(P_M^\circ).
\]
Since $k\in\{\lceil d/2\rceil-1,\ldots,d-1\}$, we have
\[
    d-1-k\in\{0,1,\ldots,\lfloor d/2\rfloor\}.
\] 
Thus, applying Theorem \ref{inscribed-vol-kfaces} to $P_M^\circ$ and using a  formula of Glasauer and Gruber \cite{glasgrub}, we obtain
\begin{equation}
    \begin{aligned}
      w(P_M)-w(B_d) &=\frac{2}{\omega_d}\int_{B_d\setminus P_M^\circ}\|x\|_2^{-(d+1)}\,dx \geq \frac{2}{\omega_d}\cdot\vol_d(B_d\setminus P_M^\circ)  \\
      &\geq \frac{2}{\omega_d}\cdot \frac{1-e^{-1}}{2}\kappa_d
    \min\left\{1,\,
    \frac d2\left(\frac{\omega_d}{4\kappa_{d-1}M}\right)^{\frac{2}{d-1}}
    \right\}\\
    &=\frac{1-e^{-1}}{d}
\min\left\{1,\frac{d}{2}\left(\frac{\omega_d}{4\kappa_{d-1}}\right)^{\frac{2}{d-1}}\MM\right\}.
    %=\frac{1-e^{-1}}{d}\min\left\{1,\frac{d}{2}\left(\frac{\omega_d}{4\kappa_{d-1}M}\right)^{\frac{2}{d-1}}\right\}.
    \end{aligned}
\end{equation}
This completes the proof. \qed
\section{Quantitative upper bounds in the large-$N$ regime}\label{sec:upper-bds}

Our next result is an upper bound for the volume difference of the Euclidean ball and an inscribed polytope with at most $N$ vertices in the large-$N$ regime.

\begin{lemma}\label{lem:explicit-inscribed-upper}
There exist absolute constants $C_0,C_1,C_2>0$ such that the following statement holds true. Let $d\geq 2$, and let $N\geq (C_0 d)^{\frac{d-1}{2}}$. Then there exists an inscribed polytope $R_N\in\mathcal{P}_N^{\rm in}(B_d)$ such that
\[
(1-C_1\NN)B_d\subset R_N\subset B_d
\qquad\text{and}\qquad
\vol_d(B_d\setminus R_N)
\leq C_2 d\kappa_d\NN.
\]
\end{lemma}

\begin{proof}
Let $r\in(0,1]$. By a standard maximal net argument, there exists a set of points $v_1,\ldots,v_m\in\mathbb{S}^{d-1}$ such that the spherical caps of geodesic radius $r$ centered at the $v_i$'s cover $\mathbb{S}^{d-1}$, and $m\leq\frac{\omega_d}{\kappa_{d-1}(r/4)^{d-1}}$. Indeed, select a maximal $r$-separated set. Then the caps of radius $r$ cover $\mathbb{S}^{d-1}$, while the caps of radius $r/2$ are pairwise disjoint. The orthogonal projection of a cap of geodesic radius $r/2$ onto the tangent hyperplane at its center contains a Euclidean $(d-1)$-ball of radius $\sin(r/2)$. Also, for $0<r\leq 1$ we have $\sin(r/2)\geq r/4$. Hence, each such cap has a $(d-1)$-dimensional volume of at least $\kappa_{d-1}(r/4)^{d-1}$, which implies the desired bound for $m$.

Now choose $r:=4\left(\frac{\omega_d}{\kappa_{d-1}N}\right)^{\frac{1}{d-1}}$. For $N\geq (C_0 d)^{(d-1)/2}$ with $C_0>0$ sufficiently large, we have $r\leq 1$, and the preceding estimate implies $m\leq N$. Let $R_N:=\operatorname{conv}\{v_1,\ldots,v_m\}$. Then $R_N\in\mathcal{P}_N^{\rm in}(B_d)$. Since the $v_i$'s form an $r$-net on the sphere, for every $u\in\mathbb{S}^{d-1}$ there exists an index $i$ such that the spherical distance between $u$ and $v_i$ is at most $r$. Therefore,
\[
h_{R_N}(u)\geq \langle u,v_i\rangle\geq\cos r.
\]
Thus, $(\cos r)B_d\subset R_N$. Using the inequality $\cos r\geq 1-r^2/2$, we obtain $(1-r^2/2)B_d\subset R_N$. 
Since $r^2=16\left(\frac{\omega_d}{\kappa_{d-1}N}\right)^{\frac{2}{d-1}}$ 
and, by Stirling's inequality,  $(\frac{\omega_d}{\kappa_{d-1}})^{\frac{2}{d-1}}$
is bounded above by an absolute constant for all $d\geq 2$, we derive that $(1-C_1\NN)B_d\subset R_N$ for some absolute constant $C_1>0$.

Finally,
\[
\vol_d(B_d\setminus R_N)
\leq\vol_d\bigl(B_d\setminus (1-r^2/2)B_d\bigr)
=\kappa_d\left(1-\left(1-\frac{r^2}{2}\right)^d\right).
\]
Since $r\leq 1$, we have $1-(1-r^2/2)^d\leq \frac{d}{2}r^2$. Hence
\[
\vol_d(B_d\setminus R_N)
\leq
\frac{d}{2}\kappa_d r^2
=8d\kappa_d\left(\frac{\omega_d}{\kappa_{d-1}N}\right)^{\frac{2}{d-1}}.
\]
Again using that $(\frac{\omega_d}{\kappa_{d-1}})^{\frac{2}{d-1}}$ 
is bounded above by an absolute constant, we derive that
\[
\vol_d(B_d\setminus R_N)
\leq C_2 d\kappa_d\NN,
\]
where $C_2>0$ is an absolute constant.
\end{proof}

\begin{remark}
In the vertices case, Theorem \ref{inscribed-vol-kfaces} and Lemma \ref{lem:explicit-inscribed-upper} are sharp up to absolute constants in the large-$N$ regime. Indeed, if $N\geq (C_0d)^{\frac{d-1}{2}}$ where $C_0>0$ is sufficiently large, then the minimum in Theorem \ref{inscribed-vol-kfaces} is attained by its second term. Hence, in the large-$N$ regime, the theorem gives $\vol_d(B_d\setminus P_N)\geq cd\kappa_d\NN$ 
for every inscribed polytope $P_N\in\mathcal{P}_N^{\rm in}(B_d)$. On the other hand, by Lemma \ref{lem:explicit-inscribed-upper} there exists an inscribed polytope $R_N\in\mathcal{P}_N^{\rm in}(B_d)$ satisfying $\vol_d(B_d\setminus R_N)
\leq Cd\kappa_d\NN$. Thus, in the inscribed vertices case $k=0$, the upper and lower bounds have the same order $d\kappa_d\NN$, 
up to absolute constants. In this sense, Theorem \ref{inscribed-vol-kfaces} and Lemma \ref{lem:explicit-inscribed-upper} are essentially sharp in the large-$N$ regime.
\end{remark}

We also obtain the following quantitative upper bound on the mean width approximation of the ball by circumscribed polytopes with at most $N$ facets.

\begin{corollary}\label{cor:explicit-circ-mw-upper}
There exist absolute constants $C_0,C>0$ such that, for all $N\geq (C_0d)^{\frac{d-1}{2}}$, there exists a circumscribed polytope $Q_N\in\mathcal{P}_N^{\rm out}(B_d)$ such that
\[
w(Q_N)-w(B_d)\leq C\NN.
\]
\end{corollary}

\begin{proof}
Let $R_N\in\mathcal{P}_N^{\rm in}(B_d)$ be the polytope from Lemma \ref{lem:explicit-inscribed-upper}, and set $Q_N:=R_N^\circ$. Then $Q_N\in\mathcal{P}_N^{\rm out}(B_d)$. By the polarity formula of Glasauer and Gruber \cite{glasgrub},
\[
w(Q_N)-w(B_d)
=\frac{2}{\omega_d}
\int_{B_d\setminus R_N}\|x\|_2^{-(d+1)}\,dx.
\]
Moreover, by Lemma \ref{lem:explicit-inscribed-upper}, we have the inclusion $(1-C_1\NN)B_d\subset R_N$. 
If $C_0>0$ is chosen to be sufficiently large, then in particular
$C_1(d+1)N^{-\frac{2}{d-1}}\leq \frac{1}{2}$. Hence, for every $x\in B_d\setminus R_N$, we get
\[
\|x\|_2^{-(d+1)}
\leq\left(1-C_1\NN\right)^{-(d+1)}
\leq C,
\]
where $C>0$ is an absolute constant. Therefore,
\[
\begin{aligned}
w(Q_N)-w(B_d)
&\leq\frac{C}{\omega_d}\cdot\vol_d(B_d\setminus R_N)
\leq\frac{C}{d\kappa_d}
\cdot d\kappa_d\NN
=C\NN.
\end{aligned}
\]
This proves the claim.
\end{proof}

\begin{remark}
In the facets case, Theorem \ref{circumscribed-mw-kfaces} and Corollary \ref{cor:explicit-circ-mw-upper} are sharp in the large-$N$ regime. Indeed, if $N\geq (C_0d)^{\frac{d-1}{2}}$ where $C_0>0$ is sufficiently large, then the minimum in Theorem \ref{circumscribed-mw-kfaces} is attained by its second term, and the theorem gives $w(Q_N)-w(B_d)\geq c\NN$
for every circumscribed polytope $Q_N\in\mathcal{P}_N^{\rm out}(B_d)$. Conversely, Corollary \ref{cor:explicit-circ-mw-upper} gives a circumscribed polytope $\widetilde{Q}_N\in\mathcal{P}_N^{\rm out}(B_d)$ such that $w(\widetilde Q_N)-w(B_d)\leq C\NN$. Therefore, in the facets case $k=d-1$, the upper and lower bounds have the same order $\NN$, up to absolute constants. In this sense, Theorem \ref{circumscribed-mw-kfaces} and Corollary \ref{cor:explicit-circ-mw-upper} are essentially sharp in the large-$N$ regime.
\end{remark}

\subsection{Extensions to convex bodies}

To conclude, we extend the upper bounds in this section to all convex bodies in $\R^d$.

\begin{theorem}\label{thm:explicit-inscribed-upper-cvx-bodies}
Let $K$ be a convex body in $\R^d$. There exist absolute constants $C_0,C_2>0$ such that the following statement holds true. Let $d\geq 2$, and let $N\geq (C_0 d)^{\frac{d-1}{2}}$. Then there exists an inscribed polytope $R_N\in\mathcal{P}_N^{\rm in}(K)$ such that
\[
\vol_d(K\setminus R_N)
\leq C_2 d\vol_d(K)\NN.
\]
\end{theorem}

\begin{proof}
    By a classical result of Macbeath \cite{Macbeath}, for every convex body $K$ in $\R^d$
    \[
\inf_{P\in\mathscr{P}_N^{\rm in}(K)}\frac{\vol_d(K\setminus P)}{\vol_d(K)} \leq \inf_{Q\in\mathscr{P}_N^{\rm in}(B_d)}\frac{\vol_d(B_d\setminus Q)}{\vol_d(B_d)} .
    \]
    Let $N\geq (C_0 d)^{\frac{d-1}{2}}$. By Lemma \ref{lem:explicit-inscribed-upper},
    \[
\inf_{Q\in\mathscr{P}_N^{\rm in}(B_d)}\frac{\vol_d(B_d\setminus Q)}{\vol_d(B_d)}\leq C_2 d\NN. 
    \]
    Therefore,
    \[
\inf_{P\in\mathscr{P}_N^{\rm in}(K)}\frac{\vol_d(K\setminus P)}{\vol_d(K)} \leq C_2 d\NN.
    \]
    Finally, by a compactness argument, there exists a best-approximating polytope $R_N\in\mathscr{P}_N^{\rm in}(K)$ achieving the infimum in the last inequality. Thus,
    \[
\vol_d(K\setminus R_N)\leq C_2 d\vol_d(K)\NN.
    \]
\end{proof}

The extension of Corollary \ref{cor:explicit-circ-mw-upper} reads as follows.

\begin{theorem}\label{thm:explicit-circ-mw-upper-cvx-bodies}
Let $K$ be a convex body in $\R^d$. There exist absolute constants $C_0,C>0$ such that, for all $N\geq (C_0d)^{\frac{d-1}{2}}$, there exists a circumscribed polytope $Q_N\in\mathcal{P}_N^{\rm out}(K)$ such that
\[
w(Q_N)-w(K)\leq Cw(K)\NN.
\]
\end{theorem}

\begin{proof}
    By a classical result of Schneider \cite{Schneider-1967}, for every convex body $K$ in $\R^d$
    \[
\inf_{P\in\mathscr{P}_N^{\rm out}(K)}\frac{w(P)-w(K)}{w(K)} \leq \inf_{Q\in\mathscr{P}_N^{\rm out}(B_d)}\frac{w(Q)-w(B_d)}{w(B_d)} .
    \]
   % The rest of the proof proceeds in the same way as that of Theorem \ref{thm:explicit-inscribed-upper-cvx-bodies}, only now we invoke Corollary \ref{cor:explicit-circ-mw-upper} rather than Lemma \ref{lem:explicit-inscribed-upper}.

    \noindent Since $w(B_d)=2$, Corollary~\ref{cor:explicit-circ-mw-upper} yields
\[
\inf_{Q\in\mathcal{P}_N^{\rm out}(B_d)}
\frac{w(Q)-w(B_d)}{w(B_d)}
\leq C\NN.
\]
Consequently,
\[
\inf_{P\in\mathcal{P}_N^{\rm out}(K)}
\bigl(w(P)-w(K)\bigr)
\leq Cw(K)\NN.
\]
To complete the proof, we choose a polytope whose error is within a factor of $2$ of this infimum, and absorbing the factor into $C$ proves the claim.
\end{proof}

\begin{comment}
\begin{remark}
The inequalities of Macbeath \cite{Macbeath} and Schneider \cite{Schneider-1967} have inspired several extensions beyond their original Euclidean setting. Geometric analogues have been established in spherical space \cite{BHLL-IUMJ-2026} and in finite-dimensional normed spaces \cite{Langi-Wang-2026}. There are also functional counterparts: extensions of Macbeath's theorem and Schneider's theorem were recently obtained in \cite{Hoehner-Novaes} and \cite{Hoehner-Mussnig}, respectively. For related extremal results in geometric and functional settings, see also, e.g.,  \cite{BucurFragalaLamboley,Chen-TAMS-2018,AFlorian1992,FlorianExtremum,FlorianPrachar,Hoehner-simplices-2026,Hoehner-studia-2026}.
\end{remark}
\end{comment}

%%%%%%%%%%%%%%%%%%%%%%%%%%%%%%

\section{Discussion}\label{sec-discussion}

The results of this paper complement the intrinsic volume approximation results obtained in \cite{HSW-2026}. In that paper, lower bounds for ordinary intrinsic volume approximation were obtained by combining endpoint estimates for vertices and facets with the combinatorial inequalities \eqref{vertices-bd} and \eqref{facets-bd}. 
The present paper shows that the same combinatorial mechanism is also useful for radial approximation and analytic dual volume deviations. The geometry, however, is different: the proofs here are based on finding a large set of directions on which the radial function of an inscribed polytope is bounded away from the radial function of the Euclidean ball.

One feature of Theorem \ref{inscribed-vol-kfaces} is the appearance of the minimum
\[
\min\left\{
1, \frac{d}{2}\left(\frac{\omega_d}{4\kappa_{d-1}}\right)^{\frac{2}{d-1}}\MM\right\}.
\]
This is unavoidable in a nonasymptotic estimate for the volume deficit of an inscribed polytope. Indeed, since $P_M\subset B_d$, the left-hand side is always bounded above by $\kappa_d$, whereas the expression $d\kappa_d\MM$ can clearly be larger than a constant multiple of $\kappa_d$ when $M$ is small relative to $d$, e.g., $M\sim Cd$ for some $C>1$ when $k=0$. Thus, the capped form of Theorem \ref{inscribed-vol-kfaces} is the natural formulation of the lower bound without a large-$M$ hypothesis.

It would be interesting to understand whether the minimum in Theorem \ref{circumscribed-mw-kfaces} is an artifact of the proof or is also the natural nonasymptotic form of the corresponding circumscribed mean width estimate. Our proof of Theorem \ref{circumscribed-mw-kfaces} proceeds by polarity from Theorem \ref{inscribed-vol-kfaces}; hence the minimum is inherited from the inscribed volume estimate. Since circumscribed polytopes can have large excess volume and large mean width excess when the number of $k$-faces is small, it is possible that a sharper direct argument could remove the capped form in Theorem \ref{circumscribed-mw-kfaces} (see also \cite{HSW-2026}).

The methods in the present paper and \cite{HSW-2026} give estimates for only one half of the $f$-vector in each polarity class. For inscribed polytopes, our results apply to $0\leq k\leq \lfloor d/2\rfloor$, because in this range the inequality $f_k(P)\geq f_0(P)$ allows one to reduce the problem to the vertex case. By polarity, this yields the corresponding circumscribed results for the range $\lceil d/2\rceil-1\leq k\leq d-1$. Hence, the remaining values of $k$ are not covered by the elementary $f$-vector inequalities \eqref{vertices-bd} and \eqref{facets-bd} used here and in \cite{HSW-2026}.

A natural open problem is therefore the following: prove lower bounds of the same order for the remaining values of $k$. More precisely, is it true that for every $k\in\{0,1,\ldots,d-1\}$ and every polytope $P_M\subset B_d$ with at most $M$ $k$-faces, we have
\[
\vol_d(B_d\setminus P_M)
\geq c\kappa_d
\min\left\{1,
d\left(\frac{\omega_d}{4\kappa_{d-1}}\right)^{\frac{2}{d-1}}\MM\right\}
\]
for some absolute constant $c>0$? More generally, does the analytic dual volume estimate of Theorem \ref{q-dual-volume-kfaces} hold for all $k\in\{0,1,\ldots,d-1\}$?

One possible route toward such a result would be to prove a radial analogue of the conjectural geometric estimate discussed in \cite{HSW-2026}. Namely, one would like to show that if an inscribed polytope has too few $k$-faces, then a positive proportion of directions on the sphere must have a radial deficit of the order $\left(\frac{\omega_d}{\kappa_{d-1}M}\right)^{\frac{2}{d-1}}$. 
Such a statement would replace the vertex covering argument used in the proof of Theorem \ref{q-dual-volume-kfaces}, and would allow one to treat the remaining part of the $f$-vector.

%Another natural direction concerns sharp asymptotics of best-approximating polytopes. In Theorem \ref{asymptotic-dual-kfaces}, we obtained lower bounds by combining Theorem 21 of \cite{BHK} with the inclusions induced by the $f$-vector inequalities \eqref{vertices-bd} and \eqref{facets-bd}. These arguments do not give matching upper bounds for intermediate $k$-faces. It would be interesting to determine whether the asymptotic constants from the vertex and facet models remain sharp for the corresponding $k$-face classes, or whether the restriction on intermediate-dimensional faces leads to different extremal behavior.

Finally, while the  estimates in this paper are proved for the Euclidean ball, a natural problem is to obtain  analogues for broader classes of convex bodies, such as those with $C^2$ boundary and everywhere positive Gauss--Kronecker curvature. In this setting, the local curvature and the radial position of the boundary should play a role analogous to the curvature-weighted expressions appearing in the asymptotic theory.

%%%%%%%%%%%%%%%%%%%%%%%%%%%%%%%%%%%
%\section{Possibly new}
\begin{comment}
{\color{blue}
Set $\widetilde{\eta}_{d,N}:=\frac{2}{3}\eta_{d,N}$. For $Q_N\in\mathcal{P}_N^{\rm out}(B_d)$, define the set
\[
\mathcal{U}_{Q_N}:=\{u\in\Sp:\,\rho_{Q_N}(u)> 1+\widetilde{\eta}_{d,N}\}.
\]
As shown in \cite{HSW-2026}, 
\begin{equation}
    \sigma(\mathcal{U}_{Q_N})\geq \frac{3}{4}. 
\end{equation}
Therefore, using also the inequality $(1+x)^j\geq 1+jx$ for $x,j\geq 0$ we get
\begin{align*}
    \widetilde{\Delta}_j(B_d,Q_N) &=\widetilde{V}_j(B_d)\int_{\Sp}\bigl(\rho_{Q_N}(u)^j-1\bigr)\,d\sigma(u)\\
    &\geq \widetilde{V}_j(B_d)\int_{\mathcal{U}_{Q_N}}\bigl(\rho_{Q_N}(u)^j-1\bigr)\,d\sigma(u)\\
    &>\widetilde{V}_j(B_d)\left[(1+\eta_{P_N})^j-1\right]\sigma(\mathcal{U}_{Q_N})\\
    &\geq \frac{3j}{4}\widetilde{V}_j(B_d)\widetilde{\eta}_{d,N}.
\end{align*}
}
\end{comment}

%%%%%%%%%%%%%%%%%%%%%%%%%%%%%%%%%%%
\section*{Acknowledgments}

Elisabeth Werner was supported by NSF grant DMS-2506790 and by Simons Fellowship SFI-MPS-SFM-00020853. She also 
wants to thank the University of M\"unster (Germany) for their hospitality. It was during her stay there as a M\"unster Research Fellow that part of the work was completed.

%%%%%%%%%%%%%%%%%%%%%%%%%%%%%%%%%
\bibliographystyle{plain}
\bibliography{main}

%%%%%%%%%%%%%%

\vspace{3mm}
\noindent
Steven Hoehner\\
\noindent {\sc Department of Mathematics \& Computer Science, Longwood University, 201 High Street, Farmville, VA 23909, U.S.A.}\\
\noindent {\it E-mail address:} {\tt hoehnersd@longwood.edu}
\vskip 3mm
\noindent
Carsten Sch\"utt\\
\noindent {\sc Department of Mathematics, University of Kiel, Heinrich-Hecht-Platz 6, 24118 Kiel, Germany} \\
\noindent {\it E-mail address:} {\tt schuett@math.uni-kiel.de} 
\vskip 3mm
\noindent
Elisabeth Werner\\
\noindent {\sc Department of Mathematics, Case Western Reserve University, 2145 Adelbert Road, Cleveland, OH 44106, U.S.A.}\\
\noindent {\it E-mail address:} {\tt elisabeth.werner@case.edu}

%%%%%%%%%%%%%%%%%%%%%%%%%%%%%%%%
\end{document}